\documentclass[12pt]{article}

\usepackage[T1]{fontenc}
\usepackage[utf8]{inputenc}
\usepackage[english]{babel}
\usepackage{url}
\usepackage{graphicx}
\usepackage{longtable}
\usepackage[fleqn]{amsmath}
\usepackage{amsthm}
\usepackage{authblk}
\usepackage{tikz}
\usetikzlibrary{arrows.meta}
\usetikzlibrary{calc}
\usetikzlibrary{ decorations.markings}
\usetikzlibrary{positioning}
\usepackage{caption}
\usepackage{subcaption}

\newtheorem{theorem}{Theorem}

\newtheorem{claim}{Claim}

\newtheorem{conjecture}[theorem]{Conjecture}

\newtheorem{problem}[theorem]{Problem}

\DeclareMathOperator{\pc}{pc}
\DeclareMathOperator{\spc}{spc}
\DeclareMathOperator{\nrc}{nrc}
\DeclareMathOperator{\Pc}{\mathcal{P}-c}

\title{Strongly proper connected coloring of graphs}
\author[1]{Micha{\l} D{\k e}bski\thanks{michal.debski87@gmail.com}}
\author[1]{Jaros{\l}aw Grytczuk \thanks{grytczuk@gmail.com (corresponding author) (Partially supported by Narodowe Centrum Nauki, grant 2020/37/B/ST1/03298.)}}
\author[1]{Pawe{\l} Naroski \thanks{pawel.naroski@pw.edu.pl}}
\author[1]{Ma\l gorzata \'{S}leszy\'{n}ska-Nowak\thanks{malgorzata.nowak@pw.edu.pl}}
\affil[1]{Faculty of Mathematics and Information Science, Warsaw University of Technology, Warsaw, Poland}

\begin{document}
	\maketitle
	\begin{abstract}
		We study a new variant of \emph{connected coloring} of graphs based on the concept of \emph{strong} edge coloring (every color class forms an \emph{induced} matching). In particular, an edge-colored path is \emph{strongly proper} if its color sequence does not contain identical terms within a distance of at most two. 
		A \emph{strong proper connected} coloring of $G$ is the one in which every pair of vertices is joined by at least one strongly proper path. Let $\spc(G)$ denote the least number of colors needed for such coloring of a graph $G$.
		
		We prove that the upper bound $\spc(G)\leq 5$ holds for any $2$-connected graph $G$. On the other hand, we demonstrate that there are $2$-connected graphs with arbitrarily large girth satisfying $\spc(G)\geq 4$. Additionally, we prove that graphs whose cycle lengths are divisible by $3$ satisfy $\spc(G)\leq 3$.
		
		We also consider briefly other connected colorings defined by various restrictions on color sequences of connecting paths. For instance, in a \emph{nonrepetitive connected coloring} of $G$, every pair of vertices should be joined by a path whose color sequence is \emph{nonrepetitive}, that is, it does not contain two adjacent identical blocks. We demonstrate that $2$-connected graphs are $15$-colorable while $4$-connected graphs are $6$-colorable, in the connected nonrepetitive sense.
		
		A similar conclusion with a finite upper bound on the number of colors holds for a much wider variety of connected colorings corresponding to fairly general properties of sequences. We end the paper with some open problems of concrete and general nature.
	\end{abstract}
	%-----------------------
	
	%linki
	
	%https://math.stackexchange.com/questions/677704/why-is-a-graph-2-connected-if-and-only-if-it-has-an-ear-decomposition
	
	%https://en.wikipedia.org/wiki/Ear_decomposition
	
	%------------------------
	
	\section{Introduction}
	Let $G$ be a simple connected graph with colored edges. A path $P$ in $G$ is \emph{proper} if no two consecutive edges of $P$ have the same color. An edge coloring of $G$ is a \emph{proper connected coloring} if every pair of distinct vertices is joined by at least one proper path. The least number of colors needed for such a coloring of a graph $G$ is denoted by $\pc(G)$.
	
	In the proper edge coloring of a graph, in a traditional sense, each pair of edges with a common vertex is colored differently, so, every path is proper. Hence, by the well known theorem of Vizing \cite{Vizing}, $$\pc(G)\leq \chi'(G) \leq \Delta(G)+1.$$ It is also easy to see that for every tree $T$, we have $\pc(T)=\chi'(T)=\Delta(T)$. However, if $G$ is a $2$-connected graph\footnote[1]{A graph $G$ is \emph{$k$-connected} if it cannot be disconnected by deleting $k-1$ vertices. Similarly, $G$ is \emph{$k$-edge-connected} if it cannot be disconnected by deleting $k-1$ edges.}, then we already have $\pc(G)\leq 3$, and this is tight. This somewhat surprising fact was proved by Borozan, Fujita, Gerek, Magnant, Manoussakis, Montero, and Tuza in \cite{Borozan-FGMM-Tuza} (see also \cite{Li-M-Qin}), where the proper connected coloring was introduced, in analogy to the well studied topic of \emph{rainbow connected coloring}, invented by Chartrand, Johns, McKeon, and Zhang in \cite{Chartrand-JM-Zhang} (see \cite{Li-Sun}). In the later concept, as you can imagine, the point is that each pair of vertices should be connected by a \emph{rainbow} path, i.e., one in which no two edges have the same color. In a similar way one may investigate a variety of coloring concepts involving various restrictions on \emph{color patterns} allowable on connecting paths (see \cite{Brause-Jendrol-Schiermeyer}, \cite{Brause-Jendrol-Schiermeyer Arxiv}, \cite{Czap-Jendrol-Valiska}).
	
	In the present paper we consider a new variant of connected coloring, inspired by the concept of \emph{strong edge coloring} of graphs, invented by Fouquet and Joliviet \cite{Fouquet-Joliviet}, and, independently, by Erd\H{o}s and Ne\v{s}et\v{r}il \cite{Erdos-Nesetril}. In the strong edge coloring of a graph $G$ every color class should form an \emph{induced} matching, which means that not only every pair of incident edges is colored differently, but also every pair of edges incident to some other common edge should be differently colored (see \cite{Deng-Yu-Zhou} for a recent survey on this topic). In particular, in a strong edge coloring of a path any sub-path with at most three edges must by rainbow. We will call such paths \emph{strongly proper}. Analogously, an edge-colored graph is called \emph{strongly proper connected} if any two vertices are connected by at least one strongly proper path. The least number of colors needed for such a coloring of a graph $G$ is denoted by $\spc(G)$ and referred to as the \emph{strong proper connection number} of $G$. 
	
	In much the same way as for the traditional edge coloring, we have the trivial bound $\spc (G)\leq \chi'_s(G)$, where $\chi'_s(G)$ is the \emph{strong chromatic index} of $G$, defined naturally as the least number of colors needed for a strong edge coloring of $G$. However, this parameter is more mysterious than its classical archetype $\chi'(G)$. In particular, a long standing conjecture by Erd\H{o}s and Ne\v{s}et\v{r}il \cite{Erdos-Nesetril} states that $\chi'_s(G)\leq\frac{5}{4}\Delta(G)^2$. This bound is tight (if true), as is demonstrated by the family of blowups of $C_5$. Currently best general result, obtained by Hurley, de Joannis de Verclos, and Kang \cite{Hurley-deVerclos-Kang}, states that the bound $\chi'_s(G)\leq 1.772\Delta(G)^2$ holds for sufficiently large $\Delta(G)$ (see \cite{Deng-Yu-Zhou} for a survey of many other results towards the Erd\H{o}s-Ne\v{s}et\v{r}il conjecture).
	
	In this paper we prove a finite upper bound on $\spc(G)$ for $2$-connected graphs. Our main result reads as follows.
	\begin{theorem}	\label{theorem_ps2connected}
		Every $2$-connected graph $G$ satisfies $\spc(G) \leq 5$.
	\end{theorem}
	
	We do not know if this upper bound is tight. Clearly, $\spc(G)\geq 3$ for any graph of diameter at least $3$. Curiously, it is not so easy to produce examples of $2$-connected graphs demanding four colors. However, we provide a general construction of a family of graphs with $\spc(G)=4$ and arbitrarily large girth. Moreover, this family contains infinitely many bipartite graphs.
	
	\begin{theorem}\label{theorem_4colors}
		For every $d\geq 3$, there exists a $2$-connected graph $G_d$ with girth at least $d$ such that $\spc(G_d) \geq 4$.
	\end{theorem}
	
	We complement these results by proving that graphs with all cycle lengths divisible by $3$ are $3$-colorable in a strongly proper connected sense.
	
	\begin{theorem}	\label{theorem_ps_3_cycles}
		Let $G$ be a $2$-connected graph. If the length of every cycle in G is divisible by $3$, then $\spc(G) \leq 3$.
	\end{theorem}
	
	The proof of this result is simple but it provides a good illustration of the main idea used in the proof Theorem \ref{theorem_ps2connected}. The main tool is based on the ear decomposition of graphs. Recall that an \emph{open ear decomposition} of a graph $G$ is a sequence $P_1,\ldots,P_h$ of subgraphs of $G$ that partition the set of edges of $G$, where $P_1$ is a cycle and every $P_i$, with $2\leq i\leq h$, is a path that intersects $P_1\cup\cdots\cup P_{i-1}$ in exactly its endpoints. Each $P_i$ is called an \emph{ear}. We will make use of following well known result of Whitney \cite{Whitney}.
	
	\begin{theorem}[Whitney, 1932]
		A graph $G$ with at least two edges is $2$-connected if and only if it has an open ear decomposition.
		\label{theorem_whitney}
	\end{theorem}
	
	A paradigm of colorful connectedness can be studied for other types of colorings as well. As postulated by Brause, Jendrol', and Schiermeyer \cite{Brause-Jendrol-Schiermeyer}, one may fix any property $\mathcal{P}$ of words (sequences), and then investigate the corresponding \emph{$\mathcal{P}$-connected coloring}, defined by the condition that every pair of vertices is joined by a path whose color sequence has the property $\mathcal{P}$. Consider for example the following, particularly intriguing property of sequences, which is close (in some sense) to the one stemming from the strong edge coloring of graphs.
	
	A sequence of the form  $c_1c_2\cdots c_nc_1c_2\cdots c_n$ is called a \emph{repetition}. An edge-colored path is \emph{nonrepetitive} if its color sequence does not contain a repetition as a \emph{block}, i.e., a subsequence of \emph{consecutive} terms. An edge-colored graph $G$ is \emph{nonrepetitively connected} if every pair of distinct vertices is joined by at least one nonrepetitive path. Denote by $\nrc(G)$ the least number of colors needed for such a coloring of $G$.
	
	Is it true that $\nrc(G)$ is finitely bounded for $2$-connected graphs? Notice that it is not at all obvious that $\nrc (P)$ is finite for every path $P$. However, by a 1906 result of Thue \cite{Thue}, every path can be nonrepetitively colored by using just \emph{three} colors. So, there is a chance for a finite bound and we prove that this is indeed true.
	
	\begin{theorem}\label{Theorem Thue}
		Every $2$-connected graph $G$ satisfies $\nrc(G)\leq 15$ and every $4$-connected graph $G$ satisfies $\nrc(G)\leq 6$.
	\end{theorem}
	 
	 Let us mention that the notion of nonrepetitive coloring of graphs, as introduced by Alon, Hałuszczak, Grytczuk, and Riordan in \cite{Alon-GH-Riordan}, can be considered more generally, in a way similar to the usual proper coloring of graphs (in both, edge or vertex version). A recent survey by Wood \cite{Wood} collects many interesting results on this topic.
	
	The proof of Theorem \ref{Theorem Thue} is based on known results on spanning trees in $k$-connected graphs. It can be easily extended to more general scenarios involving fairly universal properties of words. We discuss these issues in the final section of the paper, where we state a general conjecture on $\mathcal{P}$-connected coloring of graphs. Proofs of all our results are collected in the next section.
	
	\section{Proofs of the results}
	
	\subsection{$2$-connected graphs with cycle lengths divisible by $3$}
	We start with a simple proof that $2$-connected graphs whose all cycle lengths are divisible by $3$ satisfy $\spc(G)\leq 3$.
	\begin{proof}[Proof of Theorem \ref{theorem_ps_3_cycles}]
		Let $G$ be a $2$-connected graph. We may assume that $G$ is \emph{minimally} $2$-connected, which means that it looses this property if we remove any single edge from it. By Theorem \ref{theorem_whitney}, $G$ has an open ear decomposition, which we denote as $ED = (P_1,\ldots, P_h)$. Let $G_i$ be the subgraph of $G$ consisting of the first $i$ ears of $ED$, that is, $G_i=P_1 \cup\cdots \cup P_i$. For each ear $P_{i}$, let $s_{i}$ and $t_{i}$ be the endpoints of $P_{i}$. First we present a claim concerning the lengths of $P_i$ and some other paths in $G_i$. 
		
		\begin{claim}
			Let $P_i$ be an ear from $ED$. Then, the length of $P_i$ and the lengths of all paths between $s_i$ and $t_i$ in the graph $G_{i-1}$ are divisible by $3$ (see Figure \ref{figure_ED_3cycles}).
			\label{claim_3cycles_paths}
		\end{claim}
		\begin{proof}
			The statement is obvious for $i=1$, so, assume that $i > 1$. The graph $G_{i-1}$ is $2$-connected so there exist two paths, $Q$ and $R$, from $s_i$ to $t_i$, such that $Q \cap R = \{s_i, t_i\}$.	Let $q$ and $r$ be the lengths of $Q$ and $R$, respectively, and let $p$ be the length of $P_i$. The paths $Q$ and $R$ form a cycle in $G_{i-1}$, so $q + r$ must be divisible by $3$. Moreover, $Q$ and $P_i$, as well as $R$ and $P_i$, form a cycle in $G_i$. Therefore we have $$q + r = q + p = r + p\equiv0\pmod3.$$ This implies that $q+r+p=0\pmod3$ and therefore $ q=r=p=0\pmod3$, which completes the proof of the claim.
		\end{proof}
		
		\begin{figure}[th]
			\centering
			\begin{tikzpicture}[scale=2]
				% cykl C_12
				\foreach \i in {1, 2, 3, 4, 5, 6, 7, 8, 9, 10, 11, 12}
				{
					\coordinate (V1\i) at (90 - \i * 360/12 + 360/12:0.8);
					\path[fill=black] (V1\i) circle (0.04);
				}
				
				\foreach \i/\j in {1/2, 2/3, 3/4, 4/5, 5/6, 6/7, 7/8, 8/9, 9/10, 10/11, 11/12, 12/1}
				{
					\path[draw](V1\i) -- (V1\j);
				}
				\node[above left] at (V11) {$P_1$};
				% ucho 6 krawedzi
				\foreach \i in {1, 2, 3, 4, 5}
				{
					\coordinate (V2\i) at (95 - \i * 360/20 :1.2);
					\path[fill=black] (V2\i) circle (0.04);
				}
				\foreach \i/\j in {1/2, 2/3, 3/4, 4/5}
				{
					\path[draw](V2\i) -- (V2\j);
				}
				\path[draw](V11) -- (V21);
				\path[draw](V14) -- (V25);
				\node[above] at (V23) {$P_2$};
				
				% ucho 3 krawedzie
				\foreach \i in {1, 2}
				{
					\coordinate (V3\i) at (115 + \i * 360/10 :1.6);
					\path[fill=black] (V3\i) circle (0.04);
				}
				\path[draw](V31) -- (V32);
				\path[draw](V112) -- (V31);
				\path[draw](V19) -- (V32);
				\node[left] at ($(V31)!0.5!(V32)$) {$P_3$};
				
				% ucho 3 krawedzie
				\foreach \i in {1, 2}
				{
					\coordinate (V4\i) at (-27 + \i * 360/36 :2.2);
					\path[fill=black] (V4\i) circle (0.04);
				}
				\path[draw](V41) -- (V42);
				\path[draw](V23) -- (V42);
				\path[draw](V17) -- (V41);
				\node[right] at ($(V41)!0.5!(V42)$) {$P_4$};
				
			\end{tikzpicture}
			\caption{An exemplary ear decomposition $ED = \{P_1, P_2, P_3, P_4\}$ of some graph with cycles of length divisible by $3$ ($P_1$ is a cycle with $12$ edges, $P_2$ is a path with $6$ edges, $P_3$ and $P_4$ are paths with $3$ edges).}
			\label{figure_ED_3cycles}
		\end{figure}
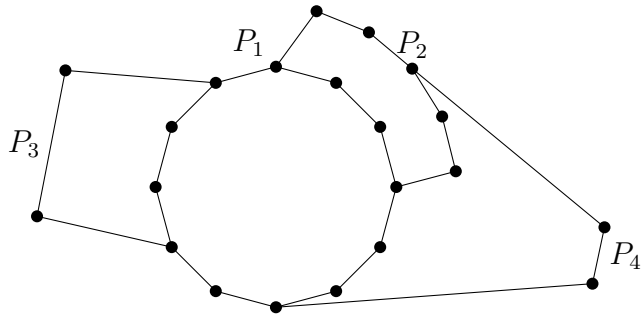
		
		Let us define a sequence of oriented graphs $D_0,D_1,\ldots,D_{h}$ such that $D_0$ is empty and for each $i>0$, $D_i$ is obtained from $D_{i-1}$ by adding the path $P_i$, oriented from $s_i$ to $t_i$. Let $D$ be the last oriented graph in this sequence, i.e., $D:=D_{h}$. Now we will color the edges of $D$ with $3$ colors. In this coloring every \emph{directed} path will be strongly proper. It is easy to see that the digraph $D$ is strongly connected, so our coloring will give us even more than the required edge coloring of $G$.
		
		Let $C$ be a $3$-coloring of the edges of $D$ by colors $\{1,2,3\}$, and let $P$ be a directed path in $D$. We say that $P$ has a \emph{canonical pattern} if the sequence of colors of $P$ forms a \emph{block} (subsequences of consecutive terms) in the infinite periodic sequence $(1, 2, 3, 1, 2, 3, 1, 2, 3, ....)$. Note that a path $P$ has a canonical pattern if and only if $P$ is strongly proper.
		
		\begin{claim}
			\label{claim_3_cycles_coloringOfD}
			There exists a $3$-coloring $C$ of the edges of $D$ such that for every restriction $C_i$ of $C$ to the subgraph $D_i$, $1\leq i\leq h$, the following properties hold:
			\begin{itemize}
				\item[(i)] For every $v\in V(D_i)$, all edges going out of $v$ have the same color.
				\item [(ii)] For every $v\in V(D_i)$, all edges going into $v$ have the same color.
				\item [(iii)] Every path in $D_i$ has a canonical pattern.
			\end{itemize}
		\end{claim}
		\begin{proof}
			We will construct the coloring $C$ by inductively defining $C_i$, for $i = 1, 2, ..., h$. The base case is when $D_1$ is a directed cycle of length divisible by $3$. We color the edges of $D_1$ consecutively with colors $1,2,3$, obtaining thereby the coloring $C_1$. The properties (i)-(iii) are obviously satisfied.
			
			Let $2\leq i \leq h$ be fixed and suppose that we have the coloring  $C_{i-1}$ of the edges of $D_{i-1}$ satisfying properties (i)-(iii). Recall that $D_i$ is obtained from $D_{i-1}$ by adding the path $P_i$ from $s_i$ to $t_i$. We will construct the coloring $C_i$ from $C_{i-1}$ by coloring the edges of $P_i$ as follows (see Figure \ref{figure_coloring_pi_3cycles}):
			\begin{itemize}
				\item [(a)] the first edge of $P_i$ has the same color as the edges going out of $s_i$ in $D_{i-1}$,
				\item [(b)] the remaining edges of $P_i$ are colored in such a way that $P_i$ has a canonical pattern (for example if the first edge has color $2$ then the second one has color $3$, the third one $1$, the fourth one $2$, and so on). We will prove that $C_i$ satisfies properties (i)-(iii).
			\end{itemize}
			
			\begin{figure}[th]
				\centering
				\begin{tikzpicture}[scale=3.75]
					
					\coordinate (V00) at (90:1);
					\coordinate (V01) at (70:1);
					\coordinate (V02) at (5:1);
					\coordinate (V03) at (-15:1);
					\coordinate (V04) at (90:0.75);
					\coordinate (V05) at (-15:0.75);
					%pomiedzy s i t
					\coordinate (V06) at (50:1);
					\coordinate (V07) at (25:1);
					
					\node[scale = 1.5, above ] at (V01) {$s_i$};
					\node[scale = 1.5, right=0.2cm] at (V02) {$t_i$};
					
					\draw[-{Latex[length=2.5mm]}, dashed] (100:1) arc (100:70:1);
					\draw[-{Latex[length=2.5mm]}, dashed] (70:1) arc (70:50:1);
					\draw[ dashed] (50:1) arc (50:25:1);
					\draw[-{Latex[length=2.5mm]}, dashed] (25:1) arc (25:5:1);
					\draw[-{Latex[length=2.5mm]}, dashed] (5:1) arc (5:-15:1);
					\draw[dashed] (-15:1) arc (-15:-25:1);
					
					\foreach \i in {0, 1, 2, 3, 4, 5, 6, 7}
					{
						\path[fill=black] (V0\i) circle (0.03);
					}
					%\path[fill=black] (0,1) circle (0.04);
					\foreach \i in {1, 2, 3, 4, 5}
					{
						\coordinate (V1\i) at (45 - \i * 360/40 + 360/20:1.5);
						\path[fill=black] (V1\i) circle (0.03);
					}
					\foreach \i/\j in {1/2, 2/3, 3/4, 4/5}
					{
						\path[-{Latex[length=2.5mm]}, draw](V1\i) -- (V1\j);
					}
					\path[-{Latex[length=2.5mm]}, draw](V01) -- (V11);
					\path[-{Latex[length=2.5mm]}, draw](V15) -- (V02);
					
					\path[draw][-{Latex[length=2.5mm]}, dashed] (V04) -- (V01);
					\path[draw][-{Latex[length=2.5mm]}, dashed] (V02) -- (V05);
					\node[above] at ($(V01)!0.5!(V11)$) {$2$};
					\node[above right] at ($(V11)!0.5!(V12)$) {$3$};
					\node[above right] at ($(V12)!0.5!(V13)$) {$1$};
					\node[above right] at ($(V13)!0.5!(V14)$) {$2$};
					\node[above right] at ($(V14)!0.5!(V15)$) {$3$};
					\node[below right] at ($(V15)!0.5!(V02)$) {$1$};
					
					\node[above] at ($(V00)!0.5!(V01)$) {$1$};
					\node[below] at ($(V04)!0.5!(V01)$) {$1$};
					\node[above] at ($(V01)!0.5!(V06)$) {$2$};
					
					\node[right] at ($(V02)!0.5!(V03)$) {$2$};
					\node[above] at ($(V02)!0.5!(V05)$) {$2$};
					\node[above right] at ($(V07)!0.5!(V02)$) {$1$};

				\end{tikzpicture}
				\caption{A coloring of $P_i$ (a fragment of the graph $D_{i-1}$ is dashed, the path $P_i$ is drawn by normal lines, the numbers $1, 2, 3$ mean colors of edges).}
				\label{figure_coloring_pi_3cycles}
			\end{figure}
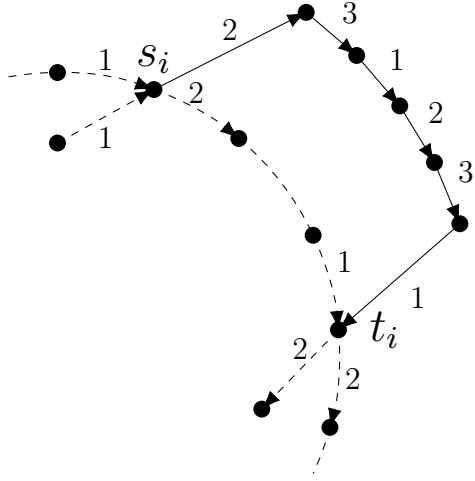
			
			Properties (i) and (ii) remain obviously satisfied for vertices from $V(D_i) \setminus V(P_i)$. They are also satisfied for all internal vertices of $P_i$, as they have both indegree and outdegree equal to $1$ in $D_i$, and for $s_i$ by (a). For $t_i$ property (i) is of course still satisfied, but to see that property (ii) holds, let us take a closer look at paths from $s_i$ to $t_i$. We know that all edges going out of $s_i$ in $D_i$ have the same color (by (i)), that all paths from $s_i$ to $t_i$ in $D_i$ has a canonical pattern (for paths from $D_{i-1}$ from (iii), for $P_i$ from (b)) and that the lengths of $P_i$ and of every path from $s_i$ to $t_i$ in $D_{i-1}$ are divisible by $3$ (from \ref{claim_3cycles_paths}). Therefore both $P_i$ and all other paths from $s_i$ to $t_i$ in $D_i$ end with the same color, so all edges going into $t_i$ have the same color and property (ii) is satisfied also for $t_i$.
			
			Now consider the property (iii). It remains clearly satisfied for paths which do not contain edges from $P_i$.
			
			Now consider some path $R$ from $u$ to $v$, where $u, v \in V(D_i)$, which contain all edges from $P_i$ (see Figure \ref{figure_r_path_case1}). It means that $R$ consists of some path $R_1$ from $u$ to $s_i$ (maybe empty), the path $P_i$ and some path $R_2$ from $t_i$ to $v$ (maybe empty). Let $Q$ be a path from $s_i$ to $t_i$ in the graph $D_{i-1}$. From (iii) we have that the path: $R_1 \cup Q \cup R_2$ has a canonical pattern, from (i) that the first edge of $P_i$ has the same color as the first edge of $Q$ and from Claim \ref{claim_3cycles_paths} that the lengths of $P_i$ and $Q$ are divisible by $3$. Therefore $R$ has a canonical pattern and property (iii) remains satisfied in this case.
			
			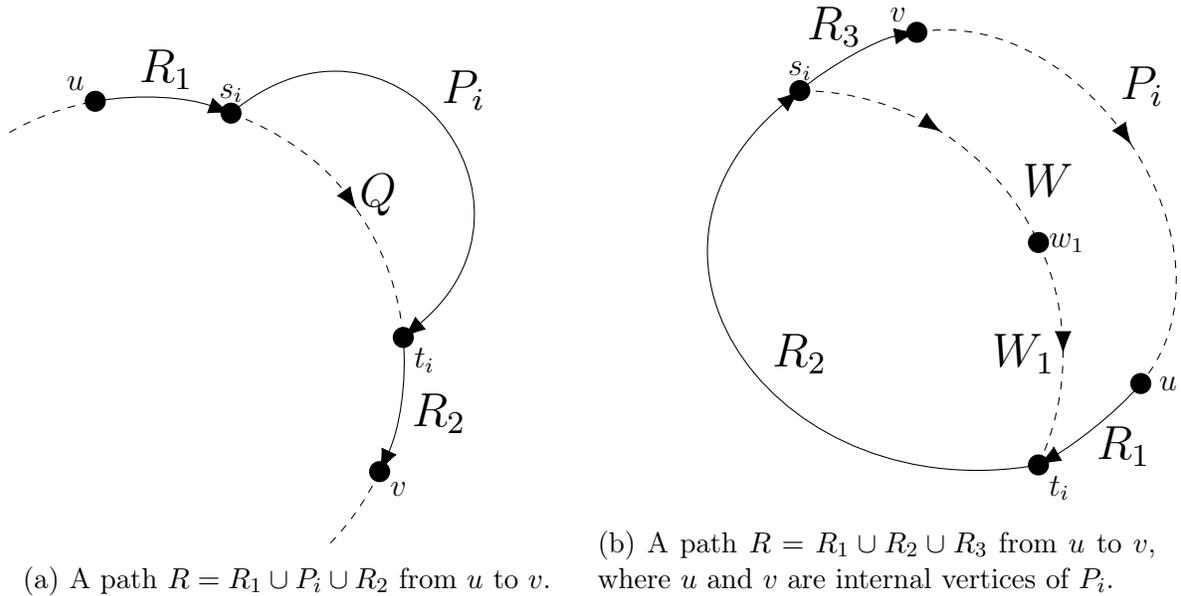
\begin{figure}[th]
				\centering
				\begin{subfigure}[b]{0.45\textwidth}
					\centering
					\begin{tikzpicture}[scale=3.5]
						
						\coordinate (V00) at (100:1);
						\coordinate (V01) at (70:1);
						\coordinate (V02) at (5:1);
						\coordinate (V03) at (-25:1);
						
						\node[above left] at (V00) {$u$};
						\node[above] at (V01) {$s_i$};
						\node[below right] at (V02) {$t_i$};
						\node[below right] at (V03) {$v$};
						\node[scale=1.5] at(85:1.11) {$R_1$};
						\node[scale=1.5] at(40:1.6) {$P_i$};
						\node[scale=1.5] at(-10:1.15) {$R_2$};
						\node[scale=1.5] at(35:1.1) {$Q$};
						
						\foreach \i in {0, 1, 2, 3}
						{
							\path[fill=black] (V0\i) circle (0.04);
						}
						
						\draw[-{Latex[length=3mm]}] (V01) to [out=40,in=40, distance=0.8cm] (V02);
						\draw[dashed] (120:1) arc (120:100:1);
						\draw[-{Latex[length=3mm]}] (100:1) arc (100:70:1);
						\draw[-{Latex[length=3mm]}, dashed] (70:1) arc (70:35:1);
						\draw[ dashed] (35:1) arc (35:5:1);
						\draw[-{Latex[length=3mm]}] (5:1) arc (5:-25:1);
						\draw[dashed] (-25:1) arc (-25:-45:1);
						
					\end{tikzpicture}
					\caption{A path $R = R_1 \cup P_i \cup R_2$ from $u$ to $v$.}
					\label{figure_r_path_case1}
				\end{subfigure}%
				\hspace{1em}% Space between image A and B
				\begin{subfigure}[b]{0.45\textwidth}
					\centering
					\begin{tikzpicture}[scale=3.5, decoration={markings, mark= at position 0.5 with {\arrow{Latex[length=3mm]}};}]
						
						\coordinate (V00) at (-5:1.3);
						\coordinate (V01) at (90:1);
						\coordinate (V02) at (-25:1);
						\coordinate (V03) at (70:1.3);
						\coordinate (V04) at (25:1);
						
						\node[right=0.1cm] at (V00) {$u$};
						\node[right] at (V04) {$w_1$};
						\node[above] at (V01) {$s_i$};
						\node[below right] at (V02) {$t_i$};
						\node[above left] at (V03) {$v$};
						\node[scale=1.5] at(-16:1.28) {$R_1$};
						\node[scale=1.5] at(38:1.65) {$P_i$};
						\node[scale=1.5] at(215:0) {$R_2$};
						\node[scale=1.5] at(84:1.25) {$R_3$};
						\node[scale=1.5] at(35:1.15) {$W$};
						\node[scale=1.5] at(0:0.85) {$W_1$};
						
						\foreach \i in {0, 1, 2, 3, 4}
						{
							\path[fill=black] (V0\i) circle (0.04);
						}
						\draw[-{Latex[length=3mm]}] (V01) to [out=40,in=190, distance=0.15cm] (V03);
						% UWAGA nowy sposób, strzałka w połowie path'a (dodane decoration przy begin{tikzpicture}
						\draw[dashed, postaction={decorate}] (V03) to [out=10,in=50, distance=0.65cm] (V00);
						\draw[-{Latex[length=3mm]}] (V00) to [out=230,in=30, distance=0.15cm] (V02);
						\draw[dashed, postaction={decorate}] (V01) arc (90:25:1);
						\draw[dashed, postaction={decorate}] (25:1) arc (25:-25:1);
						\draw[-{Latex[length=3mm]}] (V02) to [out=190,in=220, distance=1cm] (V01);

					\end{tikzpicture}
					\caption{A path $R = R_1 \cup R_2 \cup R_3$ from $u$ to $v$, where $u$ and $v$ are internal vertices of $P_i$.}
					\label{figure_r_path_case2}
				\end{subfigure}%
				\caption{The property (iii) - every path $R$ from $D_i$ has a canonical pattern.}
				\label{figure_r_path}
			\end{figure}
			
			Next let us consider a path $R$ from $u$ to $v$, where both $u$ and $v$ are internal vertices of $P_i$ (see Figure \ref{figure_r_path_case2}). First assume that $R$ is contained in $P_i$ (that is, the order of the vertices on the path $P_i$ is as follows: $s_i$, $u$, $v$, $t_i$). Then $R$ has a canonical pattern from (b). So let's assume the opposite, it means that the order of the vertices on the path $P_i$ is: $s_i$, $v$, $u$, $t_i$. Let us call the parts of $R$ as follows: let $R_1$ be a part of $R$ from $u$ to $t_i$, let $R_2$ be a part of $R$ from $t_i$ to $s_i$, and let $R_3$ be a part of $R$ from $s_i$ to $v$. Note that $R_1$ and $R_3$ are parts of $P_i$ and $R_2$ is a path from $t_i$ to $s_i$ in $D_{i-1}$. Let $W$ be some path from $s_i$ to $t_i$ in $D_{i-1}$. Let $w_1$ be a vertex from $W$ such that the distance from $w_1$ to $t_i$ on $W$ is congruent to the same value modulo $3$ as the distance from $u$ to $t_i$ on $R_i$. Let $W_1$ be a part of $W$ from $w_1$ to $t_i$. Because $P_i$ has a canonical pattern (from (b)), and also $W$ and $W_1 \cup R_2 $ have canonical patterns (from (iii)), and $R_1$ ends with the same color as $W_1$ (from (ii)), we have that the path $R_1 \cup R_2$ has a canonical pattern. Similarly, considering $w_2$ as any vertex from $V(D_{i-1})$ such that $w_2$ belongs to some path from $s_i$ to $t_i$ and the distance from $s_i$ to $w_2$ is congruent to the same value modulo $3$ as the distance from $s_i$ to $v$, we get that the path $R_2 \cup R_3$ has a canonical pattern. Therefore the whole path $R$ has a canonical pattern and property (iii) remains satisfied also in this case.
			
			The last case is when one endpoint of a path is from $V(D_{i-1})$ and the second one is an internal vertex of $P$. Carrying out analogous considerations as in the previous case, we get that the property (iii) is fulfilled here as well, which completes the proof of (iii).
			
			Therefore, the proof of the claim is complete by induction.
		\end{proof}
		
		We constructed the coloring C of the edges of $D$ such that every path in $D$ has a canonical pattern. It means that every path of $D$ is strongly proper. Therefore $G$ with coloring $C$ is strongly proper connected. The proof of Theorem \ref{theorem_ps_3_cycles} is complete.
	\end{proof}

	\subsection{$2$-connected graphs satisfy $\spc(G)\leq 5$}
	\begin{proof}[Proof of Theorem \ref{theorem_ps2connected}]
		Let $H$ be a minimally 2-connected spanning subgraph of $G$. We will construct an edge coloring of $H$ with at most $5$ colors that makes $H$ strongly proper connected. Note that it suffices to get the assertion of the theorem, as the remaining edges from $E(G) \setminus E(H)$ can be colored arbitrarily without affecting validity of the coloring.
		
		From Theorem \ref{theorem_whitney} we know that $H$ has an open ear decomposition. Similarly to the proof of Theorem \ref{theorem_ps_3_cycles} we will color the graph while adding ears, but this time the order of adding ears will be important. Let $ED = (P_1,\ldots, P_h)$ be an open ear decomposition of $H$ in which in every step we add the longest possible ear (see Figure \ref{figure_ED}). Let $H_i$ be the subgraph of $H$ consisting of the first $i$ ears of $ED$, that is, $H_i=P_1\cup\cdots \cup P_i$. For an ear $P_{i}$ let $s_{i}$ and $t_{i}$ be the endpoints of $P_{i}$.
		
		\begin{figure}[th]
			\centering
			\begin{tikzpicture}[scale=2]
				% cykl C_10
				\foreach \i in {1, 2, 3, 4, 5, 6, 7, 8, 9, 10}
				{
					\coordinate (V1\i) at (90 - \i * 360/10 + 360/10:0.8);
					\path[fill=black] (V1\i) circle (0.08);
				}
				
				\foreach \i/\j in {1/2, 2/3, 3/4, 4/5, 5/6, 6/7, 7/8, 8/9, 9/10, 10/1}
				{
					\path[draw][line width=4pt](V1\i) -- (V1\j);
				}
				
				% ucho 4 krawedzie
				\foreach \i in {1, 2, 3}
				{
					\coordinate (V2\i) at (90 - \i * 360/10 :1.6);
					\path[fill=black] (V2\i) circle (0.08);
				}
				\foreach \i/\j in {1/2, 2/3}
				{
					\path[draw][line width=3pt](V2\i) -- (V2\j);
				}
				\path[draw][line width=3pt](V11) -- (V21);
				\path[draw][line width=3pt](V15) -- (V23);
				
				% ucho 3 krawedzie
				\foreach \i in {1, 2}
				{
					\coordinate (V3\i) at (90 + \i * 360/5 :1.6);
					\path[fill=black] (V3\i) circle (0.08);
				}
				\path[draw][line width=2pt](V31) -- (V32);
				\path[draw][line width=2pt](V110) -- (V31);
				\path[draw][line width=2pt](V16) -- (V32);
				
				% ucho 2 krawedzie
				\coordinate (V41) at (90 - 2 * 360/10 :1.2);
				\path[fill=black] (V41) circle (0.08);
				\path[draw][line width=1pt](V12) -- (V41);
				\path[draw][line width=1pt](V14) -- (V41);
				
			\end{tikzpicture}
			\caption{An exemplary ear decomposition $ED = \{P_1, P_2, P_3, P_4\}$ of some graph, where the thickest lines depict $P_1$ and the thinnest lines depict $P_4$ ($P_1$ is a cycle with $10$ edges, $P_2$ is a path with $4$ edges, $P_3$ is a path with $3$ edges and $P_4$ is a path with $2$ edges).}
			\label{figure_ED}
		\end{figure}
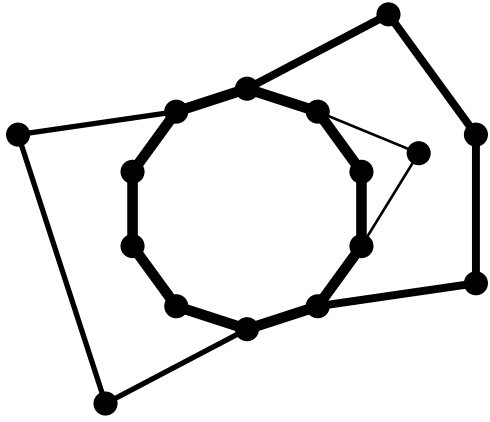

		First we present some claims concerning the properties of the ears from $ED$. Claim \ref{claim_ears} shows that we will process ears from the longest to the shortest.
		
		\begin{claim}
			The ears of $ED$ are in the opposite order to the number of their edges.  
			\label{claim_ears}
		\end{claim}
		\begin{proof}
			Let us assume the opposite. Let $P_i \in ED$ be an ear with the smallest possible index such that $P_i$ has more edges than $P_{i-1}$. Note that $i > 1$. From the definition of $ED$ we know that both endpoints of $P_i$ are internal vertices of some ears with smaller indices (maybe two different). We consider three cases.
			
			{\bf Case 1: } None of $s_i$ and $t_i$ are internal vertices of $P_{i-1}$. Therefore both $s_i$ and $t_i$ belong to some ears with indices smaller than $i-1$. It means that we could add a longer ear $P_i$ to the graph $H_{i-2}$ instead of $P_{i-1}$, which contradicts the definition of $ED$. 
			
			{\bf Case 2: } Both $s_i$ and $t_i$ are internal vertices of $P_{i-1}$. It means that instead of $P_{i-1}$ we could add to the graph $H_{i-2}$ a longer ear composed of the part of $P_{i-1}$ from one endpoint to $s_i$, the path $P_i$ and the part of $P_{i-1}$ from $t_i$ to the second endpoint of $P_{i-1}$, which again contradicts the definition of $ED$. 
			
			{\bf Case 3: } Exactly one endpoint of $P_i$ is an internal vertex of $P_{i-1}$. Assume that it is $s_i$. So $t_i$ belongs to some ear with index smaller than $i-1$. It means that to the graph $H_{i-2}$ we could add an ear composed of the part of $P_{i-1}$ from one endpoint to $s_i$ and the path $P_i$, which has more edges than $P_{i-1}$. It contradicts the definition of $ED$ again.
			
			The proof of the claim is complete.
		\end{proof}
		
		Let us recall that $H$ is minimally $2$-connected so every $P_{i}$ has at least one internal vertex. From the definition of $ED$ we know that $(s_{i}, t_{i}) =P_{i} \cap H_{i-1}$. Claim \ref{claim_endpoints_of_ears} shows that for every ear $P_{i}$ its endpoints cannot be adjacent in the graph $H_{i-1}$.
		
		\begin{claim}
			\label{claim_endpoints_of_ears}
			Vertices $s_i$ and $t_i$ are not adjacent in the graph $H_{i-1}$.
		\end{claim}
		\begin{proof}
			Let us assume the opposite. There is an edge $e=s_i t_i$ in $H_{i-1}$. Therefore there is some ear $P_j \in ED$, where $j < i$, containing $e$. According to the ordering of ears in $ED$, in every step we added the longest possible ear to our graph. But instead of $P_j$ we could add a longer ear $(P_j \setminus e) \cup P_i$, which contradicts the definition of $ED$. 
		\end{proof}
		
		Claims \ref{claim_touching_ears_the_same_number_of_edges} and \ref{claim_touching_ears_different_number_of_edges} shows the relationship between positions of the endpoints of ears with two or three edges.
		
		\begin{claim}
			\label{claim_touching_ears_the_same_number_of_edges}
			Let $P_i$ and $P_j$ be ears from $ED$ such that both have two edges or both have three edges, where $j < i$. Then no endpoint of $P_i$ can be an internal vertex of $P_j$.
		\end{claim}
		
		\begin{proof}
			Let us assume the opposite. From Claim \ref{claim_endpoints_of_ears} we know that the endpoints of any ear cannot be adjacent in $H$. Therefore the only possibility here is that exactly one endpoint of $P_i$, say $s_i$, is an internal vertex of $P_j$. Without loosing the generality of our consideration we can assume that $s_i$ is adjacent to $t_j$ in $H$. Let $r$ be some vertex from $H_{j-1}$ different than $s_j$. Let $R$ be the shortest path between $t_i$ and $r$ in the graph $H_{i-1}$ which contains edges only from $E(H_{i-1}) \setminus E(H_j)$ (see Figure \ref{figure_claim_3_ears}). Note that $R$ is an empty path if $t_i \in V(H_{j-1})$ and $R$ has at least $1$ edge in the other cases. Now consider the ear $(P_j \setminus (s_it_j)) \cup P_i \cup R$. It has more edges than $P_j$ and has both endpoints in $V(H_{j-1})$, so it could be added to the graph $H_{j-1}$. This means that in the $j$-th step we did not add the longest possible ear, which contradicts the definition of $ED$.
		\end{proof}
		
		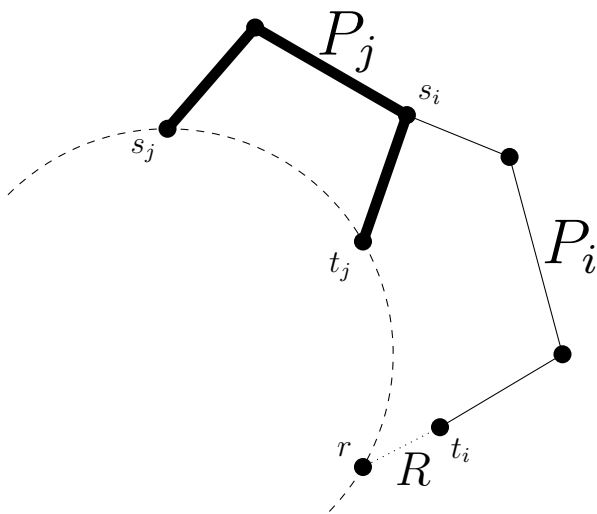
\begin{figure}[th]
			\centering
			\begin{tikzpicture}[scale=1.5]
				\draw[dashed] (135:2) arc (135:-45:2);
				
				\coordinate (V11) at (90:2);
				\coordinate (V12) at (75:3);
				\coordinate (V13) at (45:3);
				\coordinate (V14) at (30:2);
				
				\node[above right] at (V13) {$s_i$};
				\node[below left] at (V14) {$t_j$};
				\node[below left] at (V11) {$s_j$};
				\node[scale=2] at(60:3.2) {$P_j$};
				
				\foreach \i in {1, 2, 3, 4}
				{
					\path[fill=black] (V1\i) circle (0.08);
				}
				
				\foreach \i/\j in {1/2, 2/3, 3/4}
				{
					\path[draw][line width=4pt](V1\i) -- (V1\j);
				}
				
				\coordinate (V21) at (30:3.5);
				\coordinate (V22) at (0:3.5);
				\coordinate (V23) at (-15:2.5);
				\coordinate (V24) at (-30:2);
				
				\node[below right] at (V23) {$t_i$};
				\node[above left] at (V24) {$r$};
				\node[scale=2] at(15:3.7) {$P_i$};
				\node[scale=1.5] at(-25:2.4) {$R$};
				
				\foreach \i in {1, 2, 3, 4}
				{
					\path[fill=black] (V2\i) circle (0.08);
				}
				
				\foreach \i/\j in {1/2, 2/3}
				{
					\path[draw](V2\i) -- (V2\j);
				}
				\path[draw](V13) -- (V21);
				\path[draw][dotted](V23) -- (V24);
			\end{tikzpicture}
			\caption{An example of an impossible case, where one endpoint of an ear $P_i$ with $3$ edges is an internal vertex of some other ear $P_j$ with $3$ edges ($H_{j-1}$ is dashed, thick lines depict $P_j$, normal lines depict $P_i$, dotted lines depict $R$).}
			\label{figure_claim_3_ears}
			
		\end{figure}
		
		\begin{claim}
			\label{claim_touching_ears_different_number_of_edges}
			Let $P_i$ and $P_j$ be ears from $ED$ such that $j < i$, $P_i$ has two edges, $P_j$ has three edges and one endpoint of $P_i$, say $s_i$, is an internal vertex of $P_j$. Then $t_i$ (the second endpoint of $P_i$) must be an endpoint of $P_j$ not adjacent to $s_i$ in $H_{j}$.
			
		\end{claim}
		\begin{proof}
			Without loosing the generality of our consideration we can assume that $s_j$ is adjacent to $s_i$ in $H_{j}$. Note that if $t_i$ is the same as $t_j$ then our assumptions about the order of the ears in $ED$ are satisfied; see Figure \ref{figure_claim_2_3_ears}. So it remains to show that $t_i$ cannot be different from $t_j$. There are two cases to consider.
			If $t_i$ were a vertex of $P_j$ different from $t_j$, then vertices $s_i$ and $t_i$ would be adjacent in the graph $H_{i-1}$, which contradicts Claim \ref{claim_endpoints_of_ears}. If $t_i$ were from $H_{i-1} \setminus P_j$ it would mean that we added the ear $P_j$ to $H_{j-1}$ when we could add a longer ear. To construct this ear we could take the shortest path between $t_i$ and some vertex from $H_{j-1}$ different than $t_j$ which contains edges only from $E(H_{i-1}) \setminus E(H_j)$ and append to it $(P_i \cup P_j \setminus (s_is_j)))$ (the construction is analogous as in the proof of Claim \ref{claim_touching_ears_the_same_number_of_edges}). This contradicts the definition of $ED$. Therefore the proof of the claim is complete.
		\end{proof}
		
		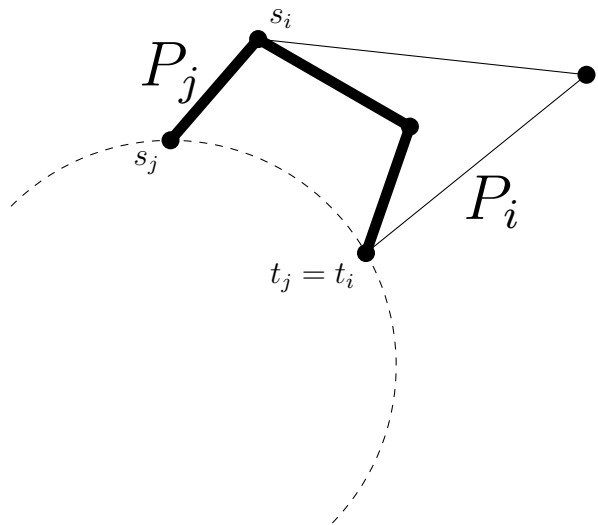
\begin{figure}[th]
			\centering
			\begin{tikzpicture}[scale=1.5]
				\draw[dashed] (135:2) arc (135:-45:2);
				
				\coordinate (V11) at (90:2);
				\coordinate (V12) at (75:3);
				\coordinate (V13) at (45:3);
				\coordinate (V14) at (30:2);
				
				\node[above right] at (V12) {$s_i$};
				\node[below left] at (V14) {$t_j = t_i$};
				\node[below left] at (V11) {$s_j$};
				\node[scale=2] at(90:2.6) {$P_j$};
				
				\foreach \i in {1, 2, 3, 4}
				{
					\path[fill=black] (V1\i) circle (0.08);
				}
				
				\foreach \i/\j in {1/2, 2/3, 3/4}
				{
					\path[draw][line width=4pt](V1\i) -- (V1\j);
				}
				
				\coordinate (V21) at (35:4.5);
				
				\node[scale=2] at(27:3.2) {$P_i$};
				
				\path[fill=black] (V21) circle (0.08);
				
				\path[draw](V12) -- (V21);
				\path[draw](V14) -- (V21);
			\end{tikzpicture}
			\caption{Ears $P_i, P_j \in ED$ such that $P_i$ has $2$ edges, $P_j$ has $3$ edges ($H_{j-1}$ is dashed, thick lines depict $P_j$ and normal lines depict $P_i$).}
			\label{figure_claim_2_3_ears}
		\end{figure}
		%Poprzednie wersje lematu
		%Let $P_i$ and $P_j$ be ears from $ED$ such that with both have $2$ edges or both have $3$ edges, where $j < i$. Then no endpoint of $P_i$ can be an internal vertex of $P_j$. 
		
		%Let $k$ be the number of edges of $P_i$. Depending on the position of the endpoints of $P_j$ there are following bounds on the number of edges of $P_j$:
		%\begin{itemize}
		%    \item[(i)] if one endpoint of $P_j$ is an internal vertex of $P_i$ and other endpoint of $P_j$ is also the endpoint of $P_i$, then  $P_j$ has at most $k-1$ edges,
		%    \item[(ii)] if both endpoints of $P_j$ are internal vertices of $P_i$, then $P_j$ has at most $k-2$ edges,
		%    \item[(iii)] if one endpoint of $P_j$ is an internal vertex of $P_i$ and other endpoint of $P_j$ does not belong to $P_i$, then $P_j$ has at most $\lfloor{\frac{k}{2}}\rfloor$ edges,
		%    \item[(iv)] otherwise $P_j$ has at most $k$ edges.
		%\end{itemize}

		Let $i^*$ be a number defined such that $P_{i^*}$ is the last ear in $ED$ with at least three edges. Let us define a sequence of oriented graphs $D_0, D_1, \ldots, D_{i^*}$ such that $D_0$ is empty and for each $i>0$, $D_i$ is obtained from $D_{i-1}$ by adding the path $P_i$, oriented from $s_i$ to $t_i$. Take $D$ to be the last oriented graph in this sequence, i.e. $D:=D_{i^*}$.
		
		Now we will color the edges of $D$ with $5$ colors such that for every two vertices $u$, $v$ there is a strongly proper \emph{directed} path from $u$ to $v$, which is a strengthening of the property required from edge-coloring of $H$. In the following claim (iii) is the crucial part, while (i), (ii) and (iv) are invariants needed in an inductive proof.
		
		\begin{claim}
			\label{claim_coloringOfD}
			There exists a coloring $C$ such that for every $i\leq i^*$, the restriction of $C_i$ to the edges of $D_i$ satisfies the following properties:
			\begin{itemize}
				\item[(i)] For every $v\in V(D_i)$, all edges going out of $v$ have the same color.
				\item [(ii)] For every $v\in V(D_i)$, all edges going into $v$ have the same color, other than the color of edges going out of $v$.
				\item [(iii)] For every two vertices $u,v\in V(D_i)$, there exists a strongly proper path from $u$ to $v$ and from $v$ to $u$.
				%\item [(iv)] Dla każdych $3$ krawędzi $xu, uv, vy \in E(D_i)$, t. że $u$ i $v$ nie są wierzchołkami wewnetrznymi tego samego ucha o $3$ krawędziach, zachodzi $C_i(xu) \neq C_i(vy)$.
				\item [(iv)] For every edge $uv\in E(D_i)$, such that $u$ and $v$ are not internal vertices of the same ear with $3$ edges, and for every two directed edges $xu, vy \in E(D_i)$, we have $C_i(xu) \neq C_i(vy)$. 
				%\item [(iv)] Jeżeli $P_i$ ma co najmniej $4$ krawędzie to każda skierowana ściezka w $D_i$ jest poprawnie silnie pokolorowana
			\end{itemize}
		\end{claim}
		\begin{proof}
			We will construct the coloring $C$ by inductively defining $C_i$ for $i=1, 2, \ldots i^*$. 
			
			The base case of the induction is when $D_1$ is a directed cycle. It suffices to color the edges of $D_1$ greedily, making sure that each edge is colored differently than the two preceding and the two following edges.
			
			Now consider any $i \leq i^*$ and suppose that we already have the coloring $C_{i-1}$ which satisfies properties (i) -- (iv) (for $i-1$). For convenience let us consider a directed path $P'_i$ obtained from $P_i$ by directing it from $s_i$ to $p_i$ and adding two edges $ws_i$ and $t_ix$ from $D_{i-1}$, for arbitrarily chosen $w$ and $x$ (see Figure \ref{figure_p_and_p_prim}).
			
			\begin{figure}[th]
				\centering
				\begin{subfigure}[b]{0.5\textwidth}
					\centering
					\begin{tikzpicture}[scale=3.5]
						
						\coordinate (V00) at (100:1);
						\coordinate (V01) at (70:1);
						\coordinate (V02) at (5:1);
						\coordinate (V03) at (-25:1);
						\coordinate (V04) at (90:0.5);
						\coordinate (V05) at (-15:0.5);
						
						\node[above = 0.1cm] at (V01) {$s_i$};
						\node[right = 0.1cm] at (V02) {$t_i$};
						\node at(40:1.65) {$P_i$};
						
						\draw[dashed] (130:1) arc (130:100:1);
						\draw[-{Latex[length=3mm]}, dashed] (100:1) arc (100:70:1);
						\draw[-{Latex[length=3mm]}, dashed] (70:1) arc (70:35:1);
						\draw[ dashed] (35:1) arc (35:5:1);
						\draw[-{Latex[length=3mm]}, dashed] (5:1) arc (5:-25:1);
						\draw[dashed] (-25:1) arc (-25:-55:1);
						
						\foreach \i in {0, 1, 2, 3, 4, 5}
						{
							\path[fill=black] (V0\i) circle (0.04);
						}
						%\path[fill=black] (0,1) circle (0.04);
						\foreach \i in {1, 2, 3, 4}
						{
							\coordinate (V1\i) at (45 - \i * 360/32 + 360/16:1.5);
							\path[fill=black] (V1\i) circle (0.04);
						}
						\foreach \i/\j in {1/2, 2/3, 3/4}
						{
							\path[draw](V1\i) -- (V1\j);
						}
						\path[draw](V01) -- (V11);
						\path[draw](V14) -- (V02);
						
						\path[draw][-{Latex[length=3mm]}, dashed] (V04) -- (V01);
						\path[draw][-{Latex[length=3mm]}, dashed] (V02) -- (V05);
						
					\end{tikzpicture}
					\caption{A fragment of the graph $D_{i-1}$ (dashed) and the path $P_i$ (normal lines).}
					\label{figure_p_and_p_prim_a}
				\end{subfigure}%
				\begin{subfigure}[b]{0.5\textwidth}
					\centering
					\begin{tikzpicture}[scale=3.5]
						
						\coordinate (V00) at (100:1);
						\coordinate (V01) at (70:1);
						\coordinate (V02) at (5:1);
						\coordinate (V03) at (-25:1);
						\coordinate (V04) at (90:0.5);
						\coordinate (V05) at (-15:0.5);
						
						\node[above = 0.1cm] at (V00) {$w$};
						\node[above =0.1cm] at (V01) {$s_i$};
						\node[right = 0.1cm] at (V02) {$t_i$};
						\node[right = 0.1cm] at (V03) {$x$};
						\node at(40:1.65) {$P'_i$};

						\draw[dashed] (130:1) arc (130:100:1);
						\draw[-{Latex[length=3mm]}, line width=3pt] (100:1) arc (100:70:1);
						\draw[-{Latex[length=3mm]}, dashed] (70:1) arc (70:35:1);
						\draw[ dashed] (35:1) arc (35:5:1);
						\draw[-{Latex[length=3mm]}, line width=3pt] (5:1) arc (5:-25:1);
						\draw[dashed] (-25:1) arc (-25:-55:1);
						
						\foreach \i in {0, 1, 2, 3, 4, 5}
						{
							\path[fill=black] (V0\i) circle (0.04);
						}
						%\path[fill=black] (0,1) circle (0.04);
						\foreach \i in {1, 2, 3, 4}
						{
							\coordinate (V1\i) at (45 - \i * 360/32 + 360/16:1.5);
							\path[fill=black] (V1\i) circle (0.04);
						}
						\foreach \i/\j in {1/2, 2/3, 3/4}
						{
							\path[draw, line width=3pt][-{Latex[length=3mm]}](V1\i) -- (V1\j);
						}
						\path[draw, line width=3pt][-{Latex[length=3mm]}](V01) -- (V11);
						\path[draw, line width=3pt][-{Latex[length=3mm]}](V14) -- (V02);
						
						\path[draw][-{Latex[length=3mm]}, dashed] (V04) -- (V01);
						\path[draw][-{Latex[length=3mm]}, dashed] (V02) -- (V05);
					\end{tikzpicture}
					\caption{The directed path $P'_i$ (thick lines).}
					\label{figure_p_and_p_prim_b}
				\end{subfigure}%
				\caption{Constructing a path $P'_i$ for a given directed graph $D_{i-1}$ and a path $P_i$. }
				\label{figure_p_and_p_prim}
			\end{figure}
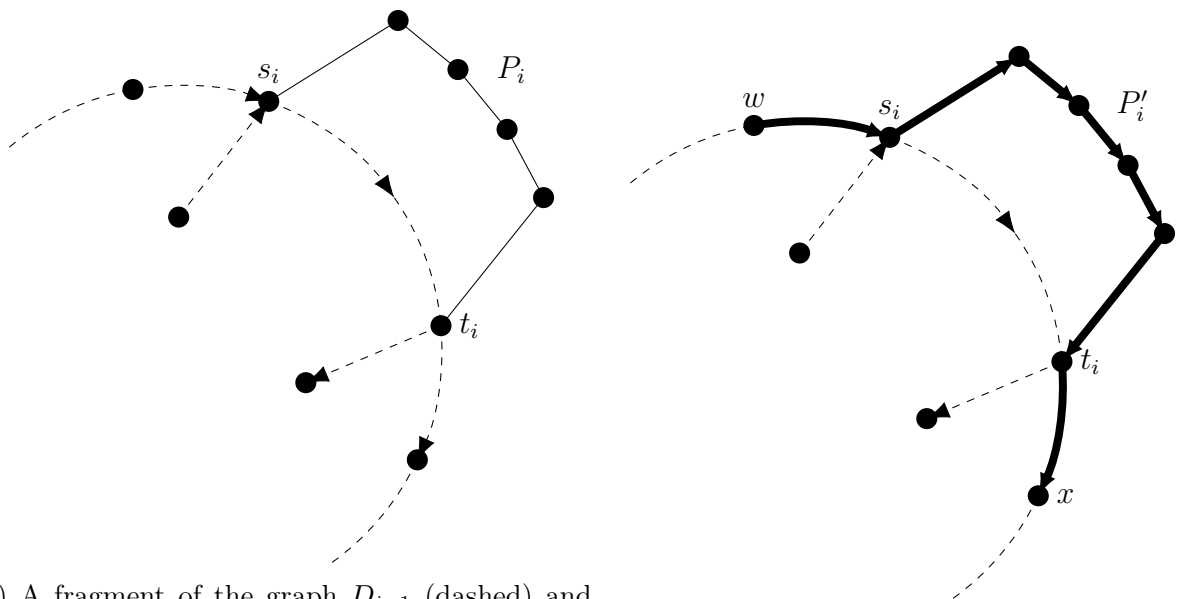
			
			Now we will construct $C_i$ from $C_{i-1}$ by coloring the internal edges of $P'_i$ such that
			\begin{itemize}
				\item [(a)] the second edge of $P'_i$ has the same color as the edges going out of $s_i$ in $D_{i-1}$,
				\item [(b)] the penultimate edge of $P'_i$ has the same color as the edges going into $t_i$ in $D_{i-1}$,
				\item [(c)] the remaining edges of $P'_i$ are colored with a different color than the two preceding and the two following edges on $P'_i$ (note that it is possible, as we have 5 colors in use).
			\end{itemize}  
			
			We will prove that $C_i$ satisfies properties (i) -- (iv).
			
			It is clear that properties (i) and (ii) remain satisfied for vertices from $V(D_{i}) \setminus V(P_i)$. They are also satisfied for $s_i$ and $t_i$ by (a) and (b), and for the other vertices of $P_i$, as they have indegree and outdegree $1$ in $D_i$.
			
			Now consider the property (iii). It remains satisfied if $u, v\in V(D_{i-1})$. If both $u$ and $v$ are internal vertices of $P_i$, then (iii) holds by (c). 
			
			Now consider the case when $u\in V(D_{i-1})$ and $v$ is an internal vertex of $P_i$. The required $u-v$ path is obtained by composing a strongly proper path $Q$ from $u$ to $s_i$ in $D_{i-1}$ (note that $Q$ exists by (iii) applied for $i-1$) with a fragment of $P_i$ from $s_i$ to $v$. We will show that this path is strongly proper. Since both $Q$ and a fragment of $P_i$ are strongly proper, we only need to exclude color conflicts between four\footnote{This statement is technically incorrect if $Q$ has one or zero edges. However, we ignore this case as it is much easier.} edges around $s_i$. First edge of $P_i$ has different color than the last edge of $Q$ by (a) and (ii). Second edge of $P_i$ is colored differently than the last edge of $Q$ by (c). For the remaining pair note that from Claim \ref{claim_touching_ears_the_same_number_of_edges} we know that $s_i$ is not an internal vertex of any ear $P_j$ with $3$ edges, where $j < i$. It follows that (iv) holds for the last edge of $Q$ in the coloring $C_{i-1}$. It follows that last but one edge of $Q$ is colored differently than edges going out of $s_i$ in $D_{i-1}$, hence by (a) there is no conflict with the first edge of $P_i$. 
			
			In the remaining case when $u$ is an internal vertex of $P_i$ and $v\in V(D_{i-1})$ the argument is analogous (the required path is obtained by composing a fragment of $P_i$ from $u$ to $t_i$ and a strongly proper path from $t_i$ to $v$). Therefore, the proof of (iii) is complete.
			
			Now consider the property (iv) that involves three edges $xu$, $uv$ and $vy$. It remains satisfied if both $u$ and $v$ are in $V(D_i)\setminus V(P_i)$. If both $u$ and $v$ are internal vertices of $P_i$, we only need to consider the case when $P_i$ has at least $4$ edges; in this case (iv) follows directly from (c).  If $u = s_i$ or $v = t_i$ the property (iv) follows directly from (c) again. Now consider the case $v=s_i$ or $u=t_i$. Let $y'$ be an out-neighbor of $v$ in $D_{i-1}$ and $x'$ be an in-neighbor of $u$ in $D_{i-1}$ (note an easy case when $x'=x$ and $y'=y$). By the induction assumption (iv) it holds that $C_{i-1}(x'u)\neq C_{i-1}(vy')$, and by (a) and (b) it follows that $C_i(vy)=C_{i-1}(vy')$ and $C_i(xu)=C_{i-1}(x'u)$, which completes the proof of (iv).
			
			Therefore, the proof of the claim is complete by induction.
		\end{proof}
		
		Before coloring the edges of $H$, we need to orient ears with two edges. Let $P_i$ be an ear with two edges. We pick $s_i'$ and $t_i'$ from $\lbrace s_i, t_i\rbrace$ such that no edge going into $s_i'$ is an internal edge of an ear with three edges and no edge going out of $t_i'$ is an internal edge of an ear with three edges. Such a choice is possible because by Claim \ref{claim_touching_ears_different_number_of_edges} if one of the vertices from $\lbrace s_i, t_i\rbrace$ is an internal vertex of some ear $P_j$ with $3$ edges, then the other one is an endpoint of $P_j$, and hence by Claim \ref{claim_touching_ears_the_same_number_of_edges} it is not an internal vertex of any other ear with $3$ edges. We will think of the ear $P_i$ as oriented from $s_i'$ to $t_i'$.
		
		Now let $C$ be a $5$-coloring of edges of $D$ given by Claim \ref{claim_coloringOfD}. Define a $5$-coloring $C'$ of edges of $H$ such that $C'(uv)=C(uv)$ for every edge $uv\in E(D)$ and for every vertex $v$ such that $v$ is an internal vertex of an ear $P_i$ with two edges, let $C'(s_i'v)$ be the color of edges going out of $s_i'$ in $C$ and $C'(vt_i')$ be the color of edges going into $t_i'$ in $C$. Note that this definition is correct, since by Claim \ref{claim_touching_ears_the_same_number_of_edges} $s_i'$ and $t_i'$ are vertices of $D$, and it is unambiguous by Claim \ref{claim_coloringOfD} (i) and (ii). We will show that $C'$ is the desired coloring of $H$.
		
		Consider any two vertices $u,v\in V(H)$. We need to find a strongly proper path between $u$ and $v$. If both $u$ and $v$ are in $V(D)$, such a path exists by Claim \ref{claim_coloringOfD} (iii). If $u,v\notin V(D)$, pick $i$ and $j$ such that $u$ is an internal vertex of the ear $P_i$ and $v$ is an internal vertex of $P_j$ (where $P_i$ and $P_j$ have two edges). In this case the desired path between $u$ and $v$ is obtained by taking a strongly proper path from $t_i'$ to $s_j'$ in $D$ (which exists by Claim \ref{claim_coloringOfD} (iii)) and appending to it the edge $ut_i'$ at the start and the edge $s_j'v$ at the end; let us denote the constructed path by $P$. We will show that $P$ is strongly proper. If $P$ has less than three edges, it follows directly from Claim \ref{claim_coloringOfD} (ii), so we assume otherwise. By Claim \ref{claim_coloringOfD} (ii) the first and second edge have different colors, and last and last but one edge on $P$ also have different colors. By the choice of $t_i'$, the second edge of $P$ is not an internal edge of any ear with $3$ edges, so by Claim \ref{claim_coloringOfD} (iv) the first and third edge on $P$ have different colors. Similarly, by the choice of $s_j'$, the last but one edge on the constructed path is not an internal edge of any ear with $3$ edges, so $P$ is strongly proper by Claim \ref{claim_coloringOfD} (iv). 
		
		Note that the same argument applies when $u\notin V(D)$ and $v\in V(D)$ or $u\in V(D)$ and $v\notin V(D)$, except that only one end of $P$ needs to be considered. This proves that $H$, together with the constructed coloring $C'$, is strongly proper connected, so the proof of Theorem \ref{theorem_ps2connected} is complete.

	\end{proof}

	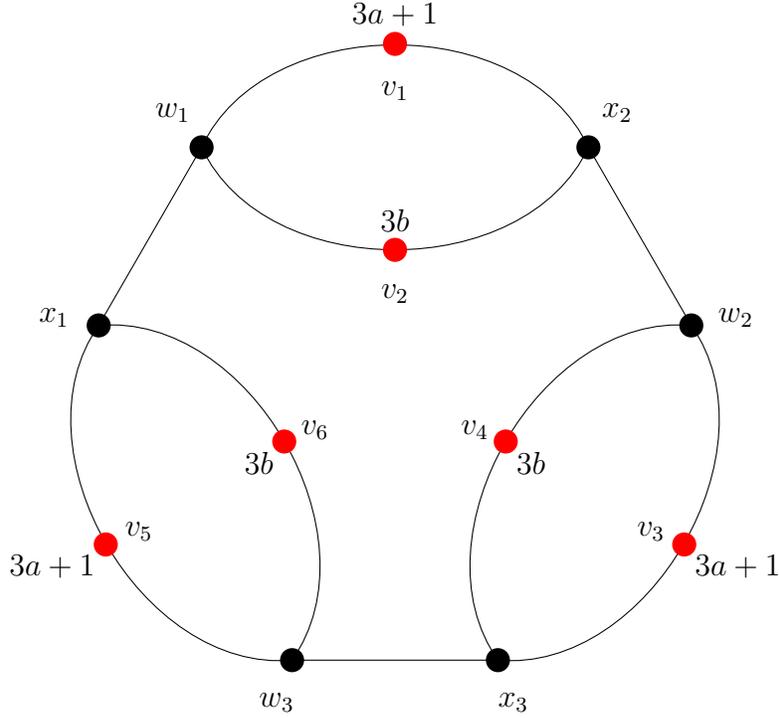
\begin{figure}[th]
		\centering
		\begin{tikzpicture}[scale=4]
			
			\coordinate (V00) at (-90 - 360/18:1);
			\coordinate (V01) at (-90 + 360/18:1);
			\coordinate (V02) at (-90 + 360/18 +360*2/9:1);
			\coordinate (V03) at (-90 + 360/18 +360*3/9:1);
			\coordinate (V04) at (-90 + 360/18 +360*5/9:1);
			\coordinate (V05) at (-90 + 360/18 +360*6/9:1);
			
			\coordinate (V00L) at (-90 - 360/18:1.15);
			\coordinate (V01L) at (-90 + 360/18:1.15);
			\coordinate (V02L) at (-90 + 360/18 +360*2/9:1.15);
			\coordinate (V03L) at (-90 + 360/18 +360*3/9:1.15);
			\coordinate (V04L) at (-90 + 360/18 +360*5/9:1.15);
			\coordinate (V05L) at (-90 + 360/18 +360*6/9:1.15);
			
			\node at (V00L) {$w_3$};
			\node at (V01L) {$x_3$};
			\node at (V02L) {$w_2$};
			\node at (V03L) {$x_2$};
			\node at (V04L) {$w_1$};
			\node at (V05L) {$x_1$};
			
			\foreach \i/\j in {0/1, 2/3, 4/5}
			{
				\path[draw](V0\i) -- (V0\j);
			}
			\draw [bend right=65] (V01) to (V02);
			\draw [bend right=-65] (V01) to (V02);
			\draw [bend right=65] (V03) to (V04);
			\draw [bend right=-65] (V03) to (V04);
			\draw [bend right=65] (V05) to (V00);
			\draw [bend right=-65] (V05) to (V00);
			
			\foreach \i in {0, 1, 2, 3, 4, 5}
			{
				\path[fill=black] (V0\i) circle (0.04);
			}
			
			\coordinate (V0) at (90:1.11);
			\coordinate (V1) at (90:0.425);
			\coordinate (V2) at (210:1.11);
			\coordinate (V3) at (210:0.425);
			\coordinate (V4) at (330:1.11);
			\coordinate (V5) at (330:0.425);
			
			\coordinate (V0R) at (90:1.11*0.8);
			\coordinate (V1R) at (90:0.425*0.5);
			\coordinate (V2R) at (210:1.11*0.8);
			\coordinate (V3R) at (210:0.425*0.5);
			\coordinate (V4R) at (330:1.11*0.8);
			\coordinate (V5R) at (330:0.425*0.5);
			
			\node[above = 0.1cm] at (V0) {$3a+1$};
			\node[above = 0.1cm] at (V1) {$3b$};
			\node[below left] at (V2) {$3a+1$};
			\node[below left] at (V3) {$3b$};
			\node[below right] at (V4) {$3a+1$};
			\node[below right] at (V5) {$3b$};
			
			% TU SA W RAZIE CZEGO TE CZERWONE WIERZCHOLKI
			\foreach \i in {0, 1, 2, 3, 4, 5}
			{
				\path[fill=red] (V\i) circle (0.04);
			}
			\node[above ] at (V0R) {$v_1$};
			\node[above ] at (V1R) {$v_2$};
			\node[below left] at (V2R) {$v_5$};
			\node[below left] at (V3R) {$v_6$};
			\node[below right] at (V4R) {$v_3$};
			\node[below right] at (V5R) {$v_4$};

		\end{tikzpicture}
		\caption{A graph $G_d$ with strong proper connection number equal to $4$.}
		\label{figure_Gd}
	\end{figure}
	
	\subsection{Proof of the lower bound $\spc (G_d)\geq 4$}
	
	\begin{proof}[Proof of Theorem \ref{theorem_4colors}]
		Let $d\geq 3$ and let $a,b\geq \max\{3,d/3\}$ be fixed integers. Consider a graph $G_d$ consisting of three edges, $x_1w_1$, $x_2w_2$, $x_3w_3$, and six paths $P_1, P_2, \ldots, P_6$ such that $P_1$ and $P_2$ go from $w_1$ to $x_2$, $P_3$ and $P_4$ go from $w_2$ to $x_3$, $P_5$ and $P_6$ go from $w_3$ to $x_1$. Moreover, $P_1, P_3, P_5$ have length $3a+1$ and $P_2, P_4, P_6$ have length $3b$ (see Figure \ref{figure_Gd}).
		
		Suppose for the contrary that $\spc(G_d)\leq 3$ and fix an edge coloring of $G_d$ with colors $1$, $2$ and $3$ that makes it strongly proper connected. Choose vertices $v_1, v_2, \ldots, v_6$ such that for each $i$, $v_i$ is a vertex from $P_i$ at distance at least $2$ from both ends of $P_i$, and the path with four edges closest to $v_i$ is strongly proper. Note that such a choice is possible since the edge-colored graph is strongly proper connected, e.g., $v_i$ is a third vertex on a strongly proper path from the central vertex of $P_i$ to $x_1$.
		
		Now define a directed graph $D$ on vertices $v_1, \ldots, v_6$ such that there is an arc $v_iv_j$ if and only if there exists a strongly proper path from $v_i$ to $v_j$ in $G_d$ with the canonical color pattern (a block of the sequence $(1,2,3,1,2,3,\ldots)$). Note that each strongly proper path that uses three colors must exhibit this canonical pattern in both directions. Since the edge-colored graph $G$ is strongly proper, it follows that $D$ contains a tournament as a directed subgraph. Now consider two cases.
		
		{\bf Case 1:} (\emph{$D$ is acyclic}) Note that in this case the vertices of $D$ can be reordered as $u_1, u_2, \ldots, u_6$ so that for $i\in \lbrace 1, 2, \ldots, 5\rbrace$ there exists a strongly proper path from $u_i$ to $u_{i+1}$ with the canonical color pattern. Note that joining all those paths produces a walk $W$ that is strongly edge-colored (i.e. no two consecutive edges and no two edges at distance $2$ on the walk have the same color) -- this follows by the choice of $v_1, v_2, \ldots, v_6$, guaranteeing that no edge on $W$ appears two times in a row. Also, each vertex from $\lbrace v_1, v_2, \ldots, v_6 \rbrace$ appears on $W$ exactly once, as otherwise there would be a cycle in $D$. This is a contradiction with the structure of $G_d$, i.e., there is no walk in $G_d$ that visits each vertex from $\lbrace v_1, v_2, \ldots, v_6 \rbrace$ exactly once and does not contain two consecutive occurrences of some edge.
		
		{\bf Case 2: } (\emph{$D$ contains a directed cycle $u_1u_2\ldots u_k$}). Consider a closed walk $W$ obtained by joining strongly proper paths with the canonical color pattern that go from $u_1$ to $u_2$, from $u_2$ to $u_3$, and so on, up to a path from $u_k$ to $u_1$. Note that, like in the previous case, $W$ is strongly edge-colored and no edge of $G_d$ appears on $W$ two times in a row. 
		%We want to show that $W$ contains a strongly edge-colored cycle.
		
		Let $S=\lbrace x_1, w_1, x_2, w_2, x_3, w_3\rbrace$ and consider consecutive occurrences of vertices from $S$ on $W$. Note that up to natural symmetries (i.e. renaming $x_i$ to $w_i$ and vice versa or rotating names, so that $x_i, w_i$ become $x_{i+1}$ and $w_{i+1}$) there are two cases: either (2a) at some point $w_1$ is followed by $x_2$, followed again by $w_1$ or (2b) $x_i$ is always followed by $w_i$, and $w_1$ is followed by $x_2$, $w_2$ by $x_3$ and $w_3$ by $x_1$.
		
		In case (2a) note that the first occurrence of $w_1$ must be preceded by $x_1$ and the second occurrence of $w_1$ -- followed by $x_1$ (because otherwise $P_1$ and $P_2$ would form a cycle with the canonical color pattern, which is impossible as their total length is not divisible by $3$). However, it implies that the edge $x_1w_1$ occurs on $W$ twice at distance exactly $3a+3b+2$, which is a contradiction, because colors on $W$ must repeat every three edges.
		
		In the remaining case (2b) a part of $W$ from the first occurrence of the edge $x_1w_1$ to its second occurrence must be a strongly edge-colored cycle; denote it by $C$. Since the length of $C$ must be divisible by $3$, it contains either vertices $\lbrace v_1, v_3, v_5 \rbrace$ or $\lbrace v_2, v_4, v_6 \rbrace$. Let $u_1, u_2, u_3$ be vertices from $\lbrace v_1, v_2, \ldots, v_6 \rbrace$ outside $C$. 
		
		Note that $u_i$ may not be incident to both incomming and outgoing arcs in $D$. Indeed, if that was the case, then a part of $C$, together with paths from $C$ to $u_i$ and from $u_i$ to $C$ with canonical color pattern, would form a strongly $3$-edge-colored cycle with length not divisible by $3$, which is a contradiction. However, this implies that $D$ does not contain an arc between two of the vertices from $\lbrace u_1, u_2, u_3\rbrace$ (i.e. there can be no arc between two vertices with outdegree $0$), which contradicts the fact that $D$ contains a tournament. Therefore, the proof is complete.
	\end{proof}
	
	\subsection{Nonrepetitive connected coloring of graphs}
	In this subsection we prove that $4$-connected graphs satisfy $\nrc(G)\leq 6$ and $2$-connected graphs satisfy $\nrc(G)\leq 15$. Actually, we will derive these bounds as simple consequences of more general results.
	
	\begin{theorem}\label{Theorem Spanning Trees}
		Let $G$ be a graph containing two edge disjoint spanning trees. Then $\nrc(G)\leq 6$.
	\end{theorem}
\begin{proof}
	Let $T_1$ and $T_2$ be two spanning trees of $G$ such that $E(T_1)\cap E(T_2)=\emptyset$. Let $r$ be a common root of these trees. Let $E_i(T_1)$ be the set of edges at distance $i$ from the root $r$. So, $E_0(T_1)$ consists of the edges of $T_1$ incident to $r$, $E_1(T_1)$ contains the edges of $T_1$ incident to the neighbors of $r$, and so on. By the theorem of Thue \cite{Thue}, there exists a nonrepetitive sequence $a_0a_1a_2\cdots$ of arbitrary length such that $a_i\in\{1,2,3\}$. We may color the edges of the tree $T_1$ using this sequence so that each edge in the set $E_i(T_1)$ gets color $a_i$. The same construction may be applied to the tree $T_2$, with similarly defined sets $E_i(T_2)$, and sufficiently long nonrepetitive sequence $b_0b_1b_2\cdots$, with $b_i\in \{4,5,6\}$. All other edges of $G$ may be colored arbitrarily.
	
	We claim that this coloring satisfies the desired property. Indeed, let $u,v$ be any two vertices of $G$. Denote by $P_j(x,y)$, $j=1,2$, the unique path from $x$ to $y$ in the tree $T_i$. Consider the path $P_1(u,r)$. Clearly, it is nonrepetitive by the construction of the coloring. If $v$ lies on $P_1(u,r)$, then the sub-path $P_1(u,v)$ is nonrepetitive, too, and we are done. So, assume that $v$ lies outside $P_1(u,r)$ and consider the path $P_2(r,v)$. If the only common vertex of these two paths is $r$, then we may glue them together into a longer path $P_1(u,r)P_2(r,v)$, which is clearly nonrepetitive, as the sets of colors on both fragments are disjoint. 
	
	Finally suppose that the two paths, $P_1(u,r)$ and $P_2(r,v)$, have some common vertices other than the root $r$ and let $x$ be the one with the largest distance from $r$ (in the tree $T_1$, say). Then the two sub-paths $P_1(u,x)$ and $P_2(x,v)$ intersect in only one vertex $x$ and, as before, we may glue them together to get the nonrepetitive path $P_1(u,x)P_2(x,v)$. This completes the proof.
\end{proof}

To get the second assertion of Theorem \ref{Theorem Thue} it suffices to apply the following simple fact following easily from the celebrated theorem of Nash-Williams \cite{Nash-Williams} (see Corollary 44 in \cite{Ozeki-Yamashita}).

\begin{theorem}[Nash-Williams \cite{Nash-Williams}]\label{Nash-Williams}
	Every $2k$-edge-connected graph contains $k$ edge-disjoint spanning trees.
\end{theorem}

Indeed, it is enough to take $k=2$ and notice that a $4$-edge-connected graph is all the more $4$-(vertex)-connected.

For the second bound for $2$-connected graphs we apply a similar approach with a silghtly weaker property based on edge independent trees. Recall that two spanning trees in a graph $G$, $T_1$ and $T_2$, having the same root $r$, are \emph{edge-independent} if for every vertex $v$, the unique paths $P_1(v,r)$ in $T_1$ and $P_2(v,r)$ in $T_2$ are edge disjoint.
		\begin{theorem}\label{Theorem Spanning Trees 2}
		Let $G$ be a graph containing two edge-independent spanning trees. Then $\nrc(G)\leq 15$.
	\end{theorem}
\begin{proof}
	Let $T_1$ and $T_2$ be two edge independent spanning trees of $G$. We will construct a similar coloring as in the proof of Theorem \ref{Theorem Spanning Trees}, but notice that this time the sets of edges $E(T_1)$ and $E(T_2)$ need not be disjoint. Therefore we will color the edges of $G$ by ordered pairs of colors whose first coordinates are controlled by an appropriate coloring of $T_1$ while second coordinates are determined by an analogous coloring of $T_2$.
	
	So, let $r$ be a common root of trees $T_1$ and $T_2$. Let $E_i(T_1)$ be the set of edges at distance $i$ from the root $r$. So, $E_0(T_1)$ consists of the edges of $T_1$ incident to $r$, $E_1(T_1)$ contains the edges of $T_1$ incident to the neighbors of $r$, and so on. By Thue's theorem \cite{Thue}, there exists a nonrepetitive sequences, $a_0a_1a_2\cdots$ and $b_0b_1b_2\cdots$, of arbitrary length such that $a_i\in\{1,2,3\}$ and $b_i\in \{4,5,6\}$. We may color the edges of trees $T_1$ and $T_2$ using these sequences so that each edge in the set $E_i(T_1)$ gets color $a_i$ and each edge in $E_i(T_2)$ gets color $b_i$. Now, if an edge $e$ belongs to both trees, then its final color is an ordered pair of colors $(a_i,b_j)$. In this way we get a partial coloring of $G$ using at most $9+6=15$ colors. The rest of the edges of $G$, not belonging to trees $T_i$, may be colored by these colors arbitrarily.
	
	We claim that this coloring satisfies the desired property. Indeed, let $u,v$ be any two vertices of $G$. Denote by $P_j(x,y)$, $j=1,2$, the unique path from $x$ to $y$ in the tree $T_i$. Consider the path $P_1(u,r)$. Clearly, it is nonrepetitive by the construction of the coloring. If $v$ lies on $P_1(u,r)$, then the sub-path $P_1(u,v)$ is nonrepetitive, too, and we are done. So, assume that $v$ lies outside $P_1(u,r)$ and consider the path $P_2(r,v)$. Let $x$ be a common vertex of these two paths, $P_1(u,r)$ and $P_2(r,v)$, such that the two sub-paths, $P_1(u,x)$ and $P_2(x,v)$ intersect only in $x$. Let $e_u$ and $e_v$ denote the last edges of these two sub-paths, respectively (so their coomon end is $x$). Now, it is not hard to verify that, by the assumption of the edge-independence of trees $T_i$, each of these two edges belong to only one tree, namely $e_u$ to $T_1$ and $e_v$ to $T_2$. Consequently, the color of $e_u$, which is some $a_i$, cannot occur at the path $P_2(v,x)$. It follows that the path $P_1(u,x)P_2(x,v)$ is nonrepetitive. This completes the proof.
\end{proof}
	It is conjectured that every $k$-edge-connected graph contains $k$ edge-independent spanning trees with an arbitrarily choosen common root $r$ (see \cite{Ozeki-Yamashita}). To get the first part of Theorem \ref{Theorem Thue} it suffices to invoke the results of Itai and Rodeh \cite{Itai-Rodeh} and Khuller and Scheiber \cite{Khuller-Scheiber} confirming this conjecture for $k=2$ (see \cite{Ozeki-Yamashita}).

	\section{Final remarks}
	Let us conclude the paper with some natural open problems. The first one is very concrete and asks for the optimum value of the strongly proper connection number $\spc(G)$ in the class of $2$-connected graphs.
	
	\begin{problem}
		Determine the least possible number $k$ such that every $2$-connected graph $G$ satisfies $\spc(G)\leq k$. 
	\end{problem} 
	We know that $k=4$ or $5$, but which is the correct value? It would be also nice to know what happens for graphs with higher connectivity. For instance, it is known (see \cite{Borozan-FGMM-Tuza}, \cite{Li-M-Qin}) that $\pc(G)=2$ holds already for $3$-connected graphs. Also, our graphs $G_d$ from the proof of Theorem \ref{theorem_4colors} are not $3$-connected. This prompts us to formulate the following conjecture.  
	
	\begin{conjecture}
		Every $3$-connected graph $G$ satisfies $\spc(G)\leq3$. 
	\end{conjecture}
	
	Let us stress however that we do not even know if the above inequality holds for graphs with any sufficiently high connectivity.
	
	It would be also nice to know more on nonrepetitive connected coloring and the corresponding parameter $\nrc(G)$.
	
	\begin{problem}
		Determine the least possible number $t$ such that every $2$-connected graph $G$ satisfies $\nrc(G)\leq t$. 
	\end{problem}
	
	By Theorem \ref{Theorem Spanning Trees 2} we know that $t\in \{3,4,\ldots, 15\}$. The problem seems challenging even if restricted to some classes of graphs. Consider, for instance, nonrepetitive connected coloring of \emph{planar} graphs. Barnette \cite{Barnette} proved that every $3$-connected planar graph contains a spanning tree of maximum degree at most three (see \cite{Ozeki-Yamashita}). Using a general upper bound form \cite{Alon-GH-Riordan} one gets that for such graphs we have $\nrc(G)\leq8$. Also, since $4$-connected planar graphs are Hamiltonian, as proved by Tutte \cite{Tutte}, they satisfy the best possible bound $\nrc(G)\leq3$.
	
	As mentioned in the introduction, one may consider fairly general \emph{$\mathcal{P}$-connected colorings}, where $\mathcal{P}$ is any property of sequences. The minimum number of colors needed for such a coloring of $G$, with fixed property $\mathcal{P}$, is denoted by $\Pc(G)$. A property $\mathcal{P}$ is called \emph{honest} if it possess the following basic features:
\begin{enumerate}
	\item [(i)] If a sequence $S$ has property $\mathcal{P}$, then each nonempty block of $S$ also satisfies $\mathcal{P}$.
	\item[(ii)] If $S$ and $T$ are two sequences over disjoint alphabets (color sets) satisfying $\mathcal{P}$, then their concatenation $ST$ also satisfies $\mathcal{P}$.
	\item[(iii)] There exist arbitrarily long sequences over some finite alphabet satisfying property $\mathcal{P}$.
\end{enumerate}
	
 Let us denote by $m(\mathcal{P})$ the least possible constant in condition (iii). For instance, if $\mathcal{P}$ corresponds to strong coloring or nonrepetitive coloring, then $m(\mathcal{P})=3$, while if $\mathcal{P}$ stems from the usual proper coloring, then $m(\mathcal{P})=2$.
 
 It is now easy to see that repeating the proof of Theorem \ref{Theorem Thue} gives the following general result.
	
	\begin{theorem}\label{Theorem P-connected}
		Let $\mathcal{P}$ be any honest property of sequences. If $G$ is a graph containing two edges disjoint spanning trees, then $\Pc(G)\leq 2m(\mathcal{P})$. In particular, every $4$-connected graph $G$ satisfies $\Pc(G)\leq 2m(\mathcal{P})$.
	\end{theorem}
	
One naturally wonders if the above statement could be true for $2$-connected graphs (or at least for $3$-connected graphs), possibly with some larger upper bound.

\begin{conjecture}
		Let $\mathcal{P}$ be any honest property of sequences. Then there exists a constant $t(\mathcal{P})$ such that every $2$-connected graph $G$ satisfies $\Pc(G)\leq t(\mathcal{P})$. 
\end{conjecture}

By the proof of Theorem \ref{Theorem Spanning Trees 2} we know that it is true if an honest property satisfies additionally the following property:
\begin{itemize}
	\item [(ii)'] If $S$ and $T$ are two sequences satisfying $\mathcal{P}$ such that the last term of $S$ does not occur in $T$ and the first term of $T$ does not apper in $S$, then their concatenation $ST$ also satisfies $\mathcal{P}$.
\end{itemize}
 
 One also naturally wonders if the minimum possible number of colors in a $\mathcal{P}$-connected coloring can be achieved at the expense of increasing connectivity.
 
  \begin{conjecture}
  	Let $\mathcal{P}$ be any honest property of sequences. Then there exists a constant $c(\mathcal{P})$ such that every $c(\mathcal{P})$-connected graph $G$ satisfies $\Pc(G)\leq m(\mathcal{P})$. 
  \end{conjecture}
 
 \section{Declarations}
 \begin{enumerate}
 	\item On behalf of all authors, the corresponding author states that there is no conflict of interest.
 	
 	\item Data sharing is not applicable to this article as no datasets were generated or analyzed during the current study.
 \end{enumerate}

\end{document}